\documentclass{article}
\usepackage[margin=1in]{geometry}
\usepackage{booktabs}
\AtEndDocument{\bigskip{\footnotesize%
  \textsc{Department of Applied Mathematics, Delhi Technological University, Delhi–110042, India} \par
  \textit{E-mail address:} \texttt{spkumar@dtu.ac.in} \par
  \addvspace{\medskipamount}
  \textsc{Department of Applied Mathematics, Delhi Technological University, Delhi–110042, India} \par
  \textit{E-mail address:} \texttt{suryagiri456@gmail.com} \par
}}
\usepackage{graphicx}
\usepackage{hyperref}
\usepackage[caption = false]{subfig}
\usepackage{amsthm}
\usepackage{amssymb}
\usepackage{caption}
\usepackage{mathrsfs}
\usepackage{mathtools}
\usepackage{enumerate}
\usepackage{amssymb}
\usepackage{wrapfig}
\usepackage{float}
\allowdisplaybreaks

\usepackage{geometry}
\numberwithin{equation}{section}
\DeclareMathOperator{\RE}{Re}
\DeclareMathOperator{\IM}{Im}
\theoremstyle{plain}
\usepackage{lipsum}

\newtheorem{theorem}{Theorem}[section]
\newtheorem{corollary}[theorem]{Corollary}

\newtheorem{lemma}{Lemma}[section]

\theoremstyle{definition}
\newtheorem{definition}[theorem]{Definition}
\theoremstyle{remark}
\newtheorem{remark}{Remark}[section]

\makeatother

\usepackage{authblk}
\begin{document}
\title{Coefficient Functional and Bohr-Rogosinski Phenomenon for Analytic functions involving Semigroup Generators}
\author{Surya Giri and S. Sivaprasad Kumar}


\date{}


	

\maketitle	
	
\begin{abstract}
	 This paper examines the coefficient problems for the class of semigroup generators, a topic in complex dynamics that has recently been studied in context of geometric function theory.
	 Further, sharp bounds of coefficient functional such as second order Hankel determinant, third order Toeplitz and Hermitian-Toeplitz determinants are derived.
	 Additionally, the sharp growth estimates and the bounds of difference of successive coefficients are determined, which are used to prove the Bohr and the Bohr- Rogosinski phenomenon for the class of semigroup generators.
\end{abstract}
\vspace{0.5cm}
\maketitle
	
\section{Introduction}
   Let $\mathcal{H}$ be the class of holomorphic functions in the unit disk $\mathbb{D}$ and $A\subset \mathcal{H}$ containing functions of the form
\begin{equation}\label{zero}
    f(z)=z +\sum_{n=2}^\infty a_n z^n.
\end{equation}
   By $\mathcal{B}$, we denote the class of holomorphic self mappings from $\mathbb{D}$ to $\mathbb{D}$.  A family $\{u_{t}(z)\}_{t\geq 0} \subset \mathcal{B}$ is called a one parameter continuous semigroup if (i) $\lim_{t\rightarrow 0} u_{t}(z)= z$,  (ii)$u_{t+s}(z) = u_{t}(z) \circ u_{s}(z)$,  and (iii) $\lim_{t\rightarrow s} u_{t}(z) = u_{s}(z)$ for each $z\in \mathbb{D}$ hold.

  Berkson and Porta \cite{BP} showed that each one parameter semigroup is locally differentiable in parameter $t\geq 0$ and moreover, if
      $$ \lim_{t\rightarrow 0} \frac{z - u_{t}(z)}{t} =f(z),$$
      which is a holomorphic function, then $u_{t}(z)$ is the solution of the
      the Cauchy problem
    $$ \frac{\partial u_t(z)}{\partial t}+ f(u_{t}(z)) =0, \quad u_{0}(z)=z.$$
    The function $f$ is called the infinitesimal (holomorphic) generator of semigroup $\{u_{t}(z)\} \subset \mathcal{B}$. The class of all holomorphic generators is denoted by $\mathcal{G}.$
     Also, note that each element of $\{u_{t}(z)\}$ generated by $f\in \mathcal{H}(\mathbb{D},\mathbb{C})$ is univalent function while $f$ is not necessarily univalent \cite{Ellin2}. Different analytic criteria are available to show that a function is semigroup generator \cite{BP,Shoi,FB}. Berkson and Porta \cite{BP} proved:
\begin{theorem}
    The following assertions are equivalent:
\begin{enumerate}
  \item[$(a)$] $f\in \mathcal{G}$;
  \item[$(b)$] $f(z)=(z - \sigma)(1- z \bar{\sigma})p(z)$ with some $\sigma \in \overline{\mathbb{D}}$ and $p\in \mathcal{H}$, $\RE(p(z))\geq 0$.
\end{enumerate}
\end{theorem}
    The point $\sigma \in \mathbb{\overline{D}}:=\{ z \in \mathbb{C}: \lvert z \rvert \leq 1\}$ is unique and called the Denjoy--Wolff point of the semigroup generated by $f(z)$. By Denjoy-Wolff theorem \cite{Ellin2,Shoi,DWT} for continuous semigroup, $\lim_{t\rightarrow \infty} u(t,z)= \sigma$ if the semigroup generated by $f$ is neither an elliptic automorphism of $\mathbb{D}$ nor an identity map for at least one $t\in [0,\infty)$.
    We denote the class of infinitesimal generators with Denjoy-Wolff point $\sigma$ by $\mathcal{G}[\sigma].$
    For $\sigma=0$, we obtain the following subclass
   $$ \mathcal{G}[0]= \{f\in  \mathcal{G} : f(z)= z p(z), \;\; \RE p(z) \geq 0 \}.$$
    Bracci et al. \cite{FB1} considered the class $\mathcal{G}_0 = \mathcal{G}[0]\cap \mathcal{A}$. In the study of non-autonomous problem such as Loewner theory, the class $\mathcal{G}_0$ plays a significant role \cite{FBL,Duren}. Various subclasses of $\mathcal{G}_0$ are considered with parameter, which is also called filtration. For instance, the class
\begin{equation}\label{one}
     \mathcal{A}_\beta = \left\{f\in \mathcal{A}:  \RE \bigg(\beta \frac{f(z)}{z}+(1-\beta) f'(z)\bigg) >0, \quad \beta\in [0,1] \right\}
\end{equation}
     is a subclass of $\mathcal{G}_0.$ In \cite{FB1}, the authors proved that $\mathcal{A}_{\beta_1} \subsetneq \mathcal{A}_{\beta_2} \subsetneq \mathcal{G}_0$ for $0 \leq \beta_1 < \beta_2 <1$ and whenever $f\in \mathcal{A}_\beta$,
     $$ \RE\frac{f(z)}{z} \geq \int_0^1 \frac{1 - t^{1-\beta}}{1 + t^{1-\beta}} dt.$$
      Clearly, when $\beta =0$, the class $\mathcal{A}_\beta$ reduces  to the class
     $ \mathcal{R} := \{ f \in \mathcal{A} : \RE f'(z) >0 \},$
     which is called the class of bounded turning functions. It can be easily seen that $\mathcal{R} \subset \mathcal{G}_0$ and each $f \in \mathcal{R}$ satisfies the Noshiro-Warshawski condition \cite{Duren}, thus every member of $\mathcal{R}$ generates a semigroup that is univalent. Elin et al. \cite{Elin} solved the radii problems for the class $\mathcal{A}_\beta$. They found the radii $r\in(0,1)$ for $f\in \mathcal{A}_\beta$ such that $f(r z)/r $ belong to the class of starlike functions, denoted by $\mathcal{S}^*$, and some other subclasses of starlike functions. This problem arises from the fact that neither $\mathcal{S}^* \subset \mathcal{A}_\beta$ nor $\mathcal{A}_\beta \subset \mathcal{S}^*$.
     Generalizing this work, Giri and Kumar \cite{Surya3} obtained $r$ such that $f(rz)/r$ belong to the class of Ma-Minda starlike functions. 


     Coefficient problems, growth estimates and others were still open for the class $\mathcal{A}_\beta$. In this paper, we focus on these problems. We find the bound of $n^{th}$ coefficient of $f\in \mathcal{A}_\beta$ and coefficient functional such as second Hankel determinant, third order Toeplitz and Hermitian Toeplitz determinant and Zalcman functional. Later, Bohr and Bohr-Rogosinski phenomenon with growth estimates are also discussed for this class.

    In 1914, Bohr \cite{Bohr} proved that,
     if $\omega(z) =\sum_{n=0}^\infty c_n z^n \in \mathcal{B}$, then $\sum_{n=0}^\infty \lvert c_n \rvert r^n \leq 1$ for all $z\in \mathbb{D}$ with $\lvert z \rvert = r \leq 1/3.$
     The constant $1/3$ is known as Bohr radius. This inequality was first derived by Bohr for $r \leq 1/6$ and sharpened independently to $r \leq 1/3$ by Wiener, Riesz, and Schur. In recent years, a lot of study have been carried out on the Bohr inequality for functions, which map $\mathbb{D}$ onto other domains, say $\Omega$ \cite{Ben,ganganina}. 
    Different generalizations of the Bohr inequality are taken into consideration \cite{Wu,Liu}. We say that
\begin{definition}
     The class $\mathcal{A}_\beta$ satisfies the Bohr phenomenon if there exists $r_b$ such that
     $$ \lvert z \rvert +\sum_{n=2}^\infty \lvert a_n \rvert \lvert z \rvert^n \leq  d( f(0), \partial f(\Omega))$$
     holds in $\lvert z \rvert = r \leq r_b,$ where $\partial f(\Omega)$ is the boundary of image domain of $\mathbb{D}$ under $f$ and $d$ denotes the Euclidean distance between $f(0)$ and $\partial f(\Omega)$.
\end{definition}
    For $m =1$, Muhanna \cite{Muhana} showed that the Bohr phenomenon holds for the class of univalent functions and the class of convex functions, when $\lvert z\rvert =r \leq 3- 2 \sqrt{2}$ and $\lvert z \rvert =r \leq 1/3$ respectively. We refer to the survey article \cite{surveryonbohr} for further details on this topic.
    There is also the concept of Rogosinski radius along with the Bohr radius, although a little is known about Rogosinski radius in comparison to Bohr radius \cite{Kamal2,Landau, Rogo}. It says that, if $\omega(z) = \sum_{n=0}^\infty c_n z^n \in \mathcal{B}$, then
    $$ \sum_{n=0}^{N-1} \lvert c_n \rvert \lvert z\rvert^n \leq 1  \quad (N \in \mathbb{N}) $$
     in the disk $\lvert z \rvert =r \leq 1/2.$ The radius $1/2$ is called the Rogosinski radius. Kayumov et al. \cite{Kayumov} considered the following expression
    $$ R_N^f(z) := \lvert f(z) \rvert + \sum_{n=N}^\infty \lvert a_n\rvert \lvert z \rvert^n $$
   and found the radius $r_N$ such that $R_N^f(z) \leq 1$ in $\lvert z \rvert = r \leq r_N$ for the Ces\'{a}ro operators on the space of bounded analytic functions. Here, we say that:
\begin{definition}
    The class $\mathcal{A}_\beta$ satisfies the Bohr-Rogosinski phenomenon if there exist $r_N$ such that
     $$ \lvert f(z^m) \rvert + \sum_{n=N}^\infty \lvert a_n \rvert \lvert z \rvert^n \leq d(f(0),\partial f(\Omega)) $$
     holds in $\lvert z \rvert = r \leq r_N$.
\end{definition}
      \noindent Section \ref{Bohr} is devoted to find the $r_b$ and $r_N$ for the class $\mathcal{A}_\beta.$

       For $f(z) = z + \sum_{n=2}^\infty a_n z^n \in \mathcal{A}$, the $m^{th}$ Hankel, Toeplitz and Hermitian Toeplitz determinant for $m\geq 1$ and $n\geq 0$ are  respectively given by
\begin{equation*}
    H_{m}(n)(f) = \begin{vmatrix}
	a_n & a_{n+1} & \cdots & a_{n+m-1} \\
	{a}_{n+1} & a_{n+2} & \cdots & a_{n+m}\\
	\vdots & \vdots & \vdots & \vdots\\
    {a}_{n+m-1} & {a}_{n+m} & \cdots & a_{n+2m-2}
    \end{vmatrix},
\end{equation*}
\begin{equation}\label{tplz}
      T_{m}(n)(f)= \begin{vmatrix}
	a_n & a_{n+1} & \cdots & a_{n+m-1} \\
	{a}_{n+1} & a_n & \cdots & a_{n+m-2}\\
	\vdots & \vdots & \vdots & \vdots\\
    {a}_{n+m-1} & {a}_{n+m-2} & \cdots & a_n
	\end{vmatrix},
\end{equation}
\begin{equation}\label{htplz}
     T_{m,n}(f)= \begin{vmatrix}
	a_n & a_{n+1} & \cdots & a_{n+m-1} \\
	\bar{{a}}_{n+1} & a_n & \cdots & a_{n+m-2}\\
	\vdots & \vdots & \vdots & \vdots\\
    \bar{{a}}_{n+m-1} & \bar{{a}}_{n+m-2} & \cdots & a_n
	\end{vmatrix},
\end{equation}
     where $\bar{a}_n = \overline{a_{n}}$.
      Toeplitz matrices have constant entries along their diagonals, while Hankel matrices have constant entries along their reverse diagonals.
      In particular, $$H_{2}(n)(f)= a_n a_{n+2} -a_{n+1}^2, \;\; T_{3}(1)(f) = 1 - 2 a_2^2 + 2 a_2^2 a_3 - a_3^2 $$ and $T_{3,1}(f) = 1 - 2 \lvert a_2 \rvert^2  +2 \RE (a_2^2 \bar{a}_3) - \lvert a_3 \rvert^2.$  Finding the sharp  bound of $\lvert H_{2}(2)(f)\rvert$ for the class $\mathcal{S}$ and its subclasses has always been the focus of many researchers.
       Although, investigations concerning Toeplitz and Hermitian Toeplitz are recently introduced in \cite{vasu,Cudna},  a summary of some of the more significant results is given in \cite{Thomas}. For more work in this direction (see \cite{Vkumar,VkumarT,Lecko2,tuneski}).

     In 1999, Ma \cite{Ma} proposed a conjecture for $f(z)= z + \sum_{n=2}^\infty a_n z^n\in \mathcal{S}$ that
     $$ \lvert J_{m,n} \rvert := \lvert a_n a_m - a_{n+m-1} \rvert \leq (m-1)(n-1).$$
     They proved this conjecture for the class of starlike functions and univalent functions with real coefficients. It is also called generalized Zalcman conjecture as it generalizes the Zalcman conjecture $\lvert a_n^2 - a_{2n-1} \rvert \leq (2n-1)^2 $ for $f\in \mathcal{S}$. Recently, bound of $\lvert J_{2,3}\rvert$  are obtained for various subclasses of $\mathcal{A}$ \cite{Allu,ChoZ}. In section \ref{H&Z} and \ref{sec3}, we obtain the sharp bound of $\lvert H_{2}(2)(f)\rvert$, $\lvert T_{3}(1)(f)\rvert$ and $\lvert J_{2,3}(f)\rvert$ for $f\in \mathcal{A}_\beta.$
\section{Hankel Determinant and Zalcman Functional }\label{H&Z}
\begin{theorem}\label{cbnd}
   If $f\in \mathcal{A}_\beta$ is of the form (\ref{zero}), then
\begin{equation}\label{two}
 \lvert a_n \rvert \leq \frac{2}{n - \beta (n-1)}.
\end{equation}
   Further, this inequality is sharp for each $n$.
\end{theorem}
\begin{proof}
   Let $f\in \mathcal{A}_\beta$ is given by (\ref{zero}), then we have
   $$ \beta \frac{f(z)}{z}+(1-\beta)z f'(z) = p(z) \quad \;\; (z\in \mathbb{D}),$$
   where $p(z) = 1+ \sum_{n=1}^\infty p_n z^n$ such that $\RE{p(z)}>0$ is a member of the Carath\'{e}odory class $\mathcal{P}$. Upon comparing the coefficients  of same powers on either side with the series expansion of $f$ and $p$ yields
\begin{equation}\label{three}
    ( n - (n-1)\beta )a_n = p_{n-1}
\end{equation}
   for $n=2,3,4,\cdots$, which gives the needed bound of $\lvert a_n \rvert$ using the Carath\'{e}odory coefficient bounds $\lvert p_n \rvert \leq 2$ (see \cite{Duren}).
    The function $\tilde{f} : \mathbb{D}\rightarrow \mathbb{C}$ defined by
\begin{equation}\label{seven}
    \tilde{f}(z)=  z \bigg( -1 + 2 \bigg( {}_2 F_1 \bigg[1, \frac{1}{1-\beta}, \frac{2- \beta}{1-\beta},z \bigg]\bigg) \bigg)= z + \sum_{n=2}^\infty \frac{2}{n -(n-1)\beta} z^n
\end{equation}   satisfies the condition $\RE \left( \beta \tilde{f}(z)/z + (1- \beta ) \tilde{f}'(z) \right)>0,$
    hence $\tilde{f}$ is a member of $\mathcal{A}_\beta$, where ${}_2 F_1$ denotes the Gauss hypergeometric function.
   Equality in (\ref{two}) occurs for $\tilde{f}$, which proves the sharpness of the bound.
\end{proof}
\begin{corollary}\label{crl11}
   If $f\in \mathcal{A}_\beta$, then for any real $\mu\geq 0$
   $$ \lvert a_n a_{n+2} - \mu a_{n+1}^2 \rvert \leq \frac{4}{(n - ( n - 1)\beta ) ( n + 2 - ( n + 1) \beta)} + \frac{4 \mu}{( n + 1- n \beta)^2}.$$
   The bound is sharp.
\end{corollary}
\begin{proof}
    Since  $\lvert a_n a_{n+2} - \mu a_{n+1}^2 \rvert \leq \lvert a_n \rvert \lvert a_{n+2} \rvert + \mu \lvert a_{n+1} \rvert^2 $. The bound simply follows from (\ref{two}).
     To see the sharpness, consider
\begin{equation}\label{tildef2}
   \tilde{f}_1(z) = z \left( -1 + 2 \left( {}_2 F_1 \left[1, \frac{1}{1-\beta}, \frac{2- \beta}{1-\beta}, i z \right]\right) \right) = z+ \sum_{n=2}^\infty \frac{2 i^{n-1}}{(n-(n-1)\beta)}  z^n.
\end{equation}
   It can be easily seen that $\tilde{f}_1(z)$ satisfy (\ref{one}), thus $\tilde{f}_1 \in \mathcal{A}_\beta$.
\end{proof}
    For $\mu =1$, Corollary \ref{crl11} gives the following sharp bound:
\begin{corollary}
    If $f\in \mathcal{A}_\beta$ is of the form (\ref{crl11}), then
    $$\lvert H_{2}(n)(f)\rvert \leq  \frac{4 \left((2 n^2 - 1)\beta^2  - (4 n^2 + 4 n - 2 )\beta + 2 n^2 + 4 n + 1 \right)}{(n - (n - 1) \beta) (n + 2 - (n + 1) \beta) (n + 1 - n \beta)^2}. $$
\end{corollary}
   For $n=2$ and 3, the following sharp bound of second order Hankel determinant follows:
\begin{corollary}
   If $f\in \mathcal{A}_\beta$ is of the form (\ref{crl11}), then
    $$\lvert H_{2}(2)(f)\rvert \leq \frac{4( 7 \beta^2 - 22 \beta + 17)}{(4 - 3 \beta) (3 - 2 \beta)^2 (2 - \beta)}, \quad \lvert H_{2}(3)(f)\rvert \leq \frac{4(17 \beta^2 - 46 \beta + 31 )}{(5 - 4 \beta) (4 - 3 \beta)^2 (3 - 2 \beta)}.$$
\end{corollary}
\begin{theorem}
   If $f\in \mathcal{A}_\beta$ is of the form (\ref{zero}), then
    $$ \lvert J_{2,3}(f) \rvert \leq \dfrac{2}{4-3 \beta}.$$
    The bound is sharp.
\end{theorem}
\begin{proof}
   Let $f \in \mathcal{A}_\beta$ is given by (\ref{zero}), then from (\ref{three}),  we have
\begin{equation}\label{five}
    \lvert J_{2,3}(f) \rvert = \lvert a_2 a_3 - a_4 \rvert = \left\lvert \frac{p_1 p_2}{(3-2 \beta)(2 - \beta)} - \frac{p_3}{4-3\beta} \right\rvert .
\end{equation}
   For $p(z)=1+ \sum_{n=1}^\infty p_n z^n \in \mathcal{P}$, Libera et al. \cite{Libera} proved that
\begin{equation}\label{libera}
\begin{aligned}
    2 p_2 &= p_1^2 + x (4 -p_1^2), \\
   4 p_3 &= p_1^3 + 2 x p_1 ( 4 - p_1^2 ) - x^2 p_1 (4 - p_1^2 ) +  2 z ( 1 - \lvert x \rvert^2 ) (4 - p_1^2),
\end{aligned}
\end{equation}
   where $\lvert x \rvert \leq 1 $ and $\lvert z \rvert \leq 1.$ Substituting these values of $p_2$ and $p_3$ in (\ref{five}), we obtain
\begin{align*}
   \lvert J_{2,3}(f) \rvert = \bigg\lvert    \frac{p_1^3}{4} \bigg(\frac{2}{2 \beta^2 - 7 \beta + 6 } + \frac{1}{ 3 \beta -4 } \bigg)  &- \frac{p_1 (4-p_1^2) (1 - \beta)^2 x}{(2 - \beta) (3 - 2 \beta) (4 - 3 \beta)} \\
       &+\frac{p_1 (4 - p_1^2) x^2}{4 (4 - 3 \beta)} -\frac{(4 - p^2) (1 - \lvert x \rvert^2) z}{2 (4 - 3 \beta)} \bigg\rvert .
\end{align*}
     Since the class $\mathcal{P}$ is rotationally invariant and it is an easy exercise to check that the class $\mathcal{A}_\beta$ is also rotationally invariant, therefore, without losing generality, we can take $p_1 =  p \in [0,2]$. Now, applying the triangle inequality with $\lvert x \rvert = \rho$, we obtain
\begin{align*}
    \lvert J_{2,3}(f) \rvert \leq    & \frac{p^3}{4} \bigg(\frac{2}{2 \beta^2 - 7 \beta + 6 } + \frac{1}{ 3 \beta -4 } \bigg)  + \frac{p (4-p^2) (1 - \beta)^2 \rho}{(2 - \beta) (3 - 2 \beta) (4 - 3 \beta)} + \frac{4 - p^2}{2 (4 - 3 \beta)}\\
       &+ \rho^2 \bigg( \frac{p (4 - p^2) }{4 (4 - 3 \beta)} - \frac{(4 - p^2)  }{2 (4 - 3 \beta)} \bigg) =: F(p,\rho) .
\end{align*}
   To determine the maximum value of $F(p,\rho)$, first we find out the stationary points, given by the roots of  $ \partial F/\partial p =0$ and $\partial F/\partial \rho =0$, where
\begin{align*}
     \frac{ \partial F(p,\rho)}{\partial p} & = \frac{3 p^2 ( r^2 (2 \beta^2- 7 \beta +6) + 4 r (1 - \beta)^2 + 2 \beta^2 - \beta  -2)}{4 (-2 + \beta) (-3 + 2 \beta) (-4 + 3 \beta)} + p \bigg( \frac{r^2}{4 - 3 \beta} - \frac{1}{4 - 3 \beta} \bigg)\\
        & \hspace{0.5cm}+ \frac{r^2}{4 - 3 \beta } + \frac{ 4 r (1 - \beta)^2}{(4 - 3 \beta) (3 - 2 \beta) (2 - \beta)}.\\
     \frac{ \partial F(p,\rho)}{\partial \rho} &= 2 r \bigg( \frac{ p (4 - p^2)}{4 (4 - 3 \beta)} + \frac{-4 + p^2 }{ 2 (4 - 3 \beta)} \bigg)+ \frac{p (4 - p^2) (1 - \beta)^2}{(4 - 3 \beta) (3 -2 \beta) (2 - \beta)}. \\
\end{align*}
   A simple calculation shows that for $p\in [0,2]$ and $r\in [0,1]$, the stationary point is $(0,0)$ and
   $$ \bigg(\frac{\partial^2 F}{\partial p^2} \frac{\partial^2 F}{\partial \rho^2} - \frac{\partial^2 F}{\partial \rho \partial p}\bigg)_{(p,\rho)=(0,0)} = \frac{4 (8 - 11 \beta + 4 \beta^2)}{(3 -
    2 \beta)^2 (2 - \beta)^2 (4 - 3 \beta))}>0 \;\;\; \text{for all} \;\; \beta \in [0,1].$$
    Thus $F(p,\rho)$ attains either maximum or minimum at $(p,\rho)=(0,0)$. Since, we have
   $$  \bigg(\frac{\partial^2 F}{\partial p^2}\bigg)_{(0,0)} = \frac{-1}{4 - 3 \beta}<0 , \;\;  \bigg(\frac{\partial^2 F}{\partial \rho^2}\bigg)_{(0,0)} = \frac{-4}{4 - 3 \beta} < 0  \;\;\; \text{for all} \;\;\; \beta \in [0,1].$$
   Therefore, $F(p,\rho)$ attain its maximum value at $(p,\rho)=(0,0)$, which is  $ 2/(4 - 3 \beta).$

   Now, to prove the sharpness of the bound, consider the function $\tilde{f}_2 : \mathbb{D} \rightarrow \mathbb{C}$ given by
\begin{equation}\label{six}
   \alpha \frac{\tilde{f}_2(z)}{z}+(1-\alpha) \tilde{f}_2'(z) = \frac{1+z^3}{1-z^3}.
\end{equation}
   If $\tilde{f}_2 (z) = z+ \sum_{n=2}^\infty a_n z^n$, then   $a_2= a_3 =0$ and $a_4 =2/(4 - 3 \beta)$, thus $\lvert J_{2,3}(f) \rvert = 2/(4 - 3 \beta)$.
\end{proof}
\section{Toeplitz and Hermitian-Toeplitz Determinant}\label{sec3}
\begin{theorem}\label{fthm}
   If $f\in \mathcal{A}_\beta$ is of the form (\ref{zero}), then
\begin{enumerate}[(i)]
  \item $ \lvert T_{2,n}(f)\rvert \leq  4 \bigg(  \dfrac{1}{( n - \beta (n-1) )^2}  + \dfrac{1}{( n + 1- n \beta  )^2} \bigg), $
  \item $ \lvert T_{3,1}(f)\rvert  \leq \dfrac{4 \beta^4  - 28 \beta^3 + 101 \beta^2 - 196 \beta  + 140
 }{(3 - 2 \beta)^2 (\beta -2 )^2}.$
\end{enumerate}
   The bounds are sharp.
\end{theorem}
\begin{proof}
   From (\ref{tplz}), it follows that
\begin{align*}
    \lvert T_{2,n}(f)\rvert = \lvert a_{n}^2 - a_{n+1}^2 \rvert &\leq \lvert a_{n} \rvert^2 + \lvert a_{n+1} \rvert^2 .
\end{align*}
   Using the bound of $\lvert a_n \rvert$ from (\ref{two}), required bound of $\lvert T_{2,n}(f)\rvert$ follows directly
    and equality case holds for the function $\tilde{f}_1$ given by (\ref{tildef2}).

    Now we proceed for $\lvert T_{3,1}(f)\rvert$. Again from (\ref{tplz}), we have
\begin{equation}\label{four}
    \lvert T_{3,1}(f) \rvert =\lvert 1 - 2 a_2^2 +2 a_2^2 a_3 - a_3^2 \rvert \leq 1 + 2 \lvert a_2 \rvert^2 + \lvert a_3 \rvert \lvert a_3 - 2 a_2^2 \rvert.
\end{equation}
   By (\ref{three}),
   $$ \lvert a_3 - 2 a_2^2 \rvert = \frac{1}{3- 2 \beta} \left\lvert p_2 - \frac{2 (3 - 2 \beta)}{(2 - \beta)^2} p_1^2 \right\rvert. $$
   Applying the  well known result $\lvert p_2 - \mu p_1^2 \rvert \leq 4 \mu -2$ for $\mu >1$ (see \cite{MaMinda}), we obtain
   $$ \lvert a_3 - 2 a_2^2 \rvert  \leq \frac{8}{(2- \beta)^2} - \frac{2}{3 - 2 \beta}.$$
   Using this bound of $\lvert a_3 - 2 a_2^2 \rvert$ and the bounds of $\lvert a_2 \rvert$, $\lvert a_3 \rvert$ from (\ref{two}) in (\ref{four}), required bound of $\lvert T_{3,1}(f) \rvert$ follows. Sharpness of the bound of $\lvert T_{3,1}(f) \rvert$ follows from the function $\tilde{f}_1$.
\end{proof}
\begin{remark}
   The bounds of $\lvert T_{2,n}(f)\rvert$ and $\lvert T_{3,1}(f)\rvert$ for the class $\mathcal{R}$ follow from Theorem \ref{fthm}, when $\beta =0$ \cite[Theorem 2.12]{vasu}.
\end{remark}
\begin{theorem}\label{LHT}
    If $f\in \mathcal{A}_\beta$ is of the form (\ref{zero}), then
\begin{equation}\label{eight}
\begin{aligned}
     {T_{3,1}(f)}\leq \left\{
    \begin{array}{ll}
    \dfrac{4 \beta^4 - 28 \beta^3 + 37 \beta^2 - 4 \beta -4 }{(3 - 2 \beta)^2 (2 - \beta)^2}; &\frac{10-\sqrt{10}}{9}  \leq \beta \leq 1, \\  \\
    1; & 0\leq \beta \leq \frac{10-\sqrt{10}}{9}.
    \end{array}
    \right.
\end{aligned}
\end{equation}
   The bounds are sharp.
\end{theorem}
\begin{proof}
     For $f(z) = z+ \sum_{n=2}^\infty a_n z^n \in \mathcal{A}_\beta $, Theorem \ref{cbnd} yields
    $$ \lvert a_2 \rvert \leq \frac{2}{2-\beta} \;\; \text{and} \;\; \lvert a_3 \rvert \leq \frac{2}{3 - 2 \beta}.$$
    Hence $\lvert a_2 \rvert \in [0, 2]$ and $\lvert a_3 \rvert \in [0, 2]$ for $\beta \in [0,1].$
     From (\ref{htplz}), we have
\begin{align*}
      {T}_{3,1}(f) &= 1 + 2 \RE (a_2^2 \bar{a}_3) - 2 \lvert a_2 \rvert^2 - \lvert a_3 \rvert^2 \\
                        &\leq  1 + 2  \lvert a_2 \rvert^2 \lvert a_3\rvert - 2 \lvert a_2 \rvert^2 - \lvert a_3 \rvert^2 =: g(\lvert a_3 \rvert),
\end{align*}
    where $g(x) =  1 + 2  \lvert a_2 \rvert^2 x - 2 \lvert a_2 \rvert^2 - x^2$ with $x= \lvert a_3 \rvert \in [0,2]$. Since
    $ g'(x)=0 $ at $x_0:= \lvert a_2 \rvert^2$ and $g''(x_0)<0$, therefore $g(x)$ attains its maximum value at $x = x_0$, whenever $\lvert a_2 \rvert^2$ belongs to the range of $x$, that means $\lvert a_2 \rvert^2  \leq 2$. Thus
\begin{align*}
     {T}_{3,1}(f) \leq g(\lvert a_2\rvert^2) &= (\lvert a_2 \rvert^2 -1)^2 \\
                                                &\leq   1  \quad \text{when} \quad \lvert a_2 \rvert^2 \leq 2, \\
      & = 1 \quad \text{when} \quad 0\leq \beta \leq 2 -\sqrt{2}.
\end{align*}
     Now, the other case, when $\lvert a_2 \rvert^2$ does not lie in the range of $x$, that is $ \lvert a_2 \rvert^2 > 2$ or $ 2 -\sqrt{2} \leq \beta \leq 1$, then
\begin{align*}
    {T}_{3,1}(f) &\leq \max{g(x)}  = g \bigg( \frac{2}{3- 2 \beta}\bigg)\\
                     &= 1 - 2 a_2^2 - \frac{4}{(3 - 2 \beta)^2} + \frac{4 a_2^2}{3 - 2 \beta} \\
                     &\leq \frac{4 \beta^4 - 28 \beta^3 + 37 \beta^2 - 4 \beta -4 }{(3 - 2 \beta)^2 (2 - \beta)^2}.
\end{align*}
    Using all these above arguments, we obtain
\begin{align*}
     {T_{3,1}(f)}\leq \left\{
    \begin{array}{ll}
    1, & 0\leq \beta \leq \beta_0; \\
    \frac{4 \beta^4 - 28 \beta^3 + 37 \beta^2 - 4 \beta -4 }{(3 - 2 \beta)^2 (2 - \beta)^2}, & \beta_0 \leq \beta \leq 1,
    \end{array}
    \right.
\end{align*}
   where $\beta_0 = (10 - \sqrt{10})/9$ is the root of the equation $9 \beta^2 - 20 \beta +10 =0 .$

   The sharpness of the bound follows from $f(z)=z$ when $0\leq \beta \leq ({10-\sqrt{10}})/{9}$. However, for $({10-\sqrt{10}})/{9} \leq \beta \leq 1$, equality in (\ref{eight}) holds for the function $\tilde{f}$ given in (\ref{seven}).
\end{proof}
\begin{remark}
     For $\beta =0$ in Theorem \ref{LHT}, we obtain $ T_{3,1}(f) \leq 1$ for $f\in \mathcal{R}$ \cite[Example 2.4]{VkumarT}.
\end{remark}
\begin{theorem}
    If $f\in \mathcal{A}_\beta$ is of the form (\ref{zero}), then
    $$  T_{3,1}(f)\geq  1 - \frac{ 4 \beta -9}{\beta^4 - 4 \beta^3 + 2 \beta^2 + 8 \beta -8}.$$
    The bound is sharp.
\end{theorem}
\begin{proof}
     Let $f \in \mathcal{A}_\beta$, then from (\ref{three}), we have
     $$ a_2 =\frac{p_1}{2 - \beta},\;\;\;  a_3 = \frac{p_2}{3 - 2 \beta}.$$
     Now, by replacing $p_2$ in terms of $p_1$ using (\ref{libera}), we get
\begin{align*}
   2 \RE (a_2^2 \bar{a}_3) &= \frac{p_1^4 + p_1^2 (4 - p_1^2) \RE{\zeta}}{(3 - 2 \beta) (2 - \beta)^2},  \quad - \lvert a_2 \rvert^2 = \frac{\lvert p_1 \rvert^2}{(2-\beta)^2}, \\
    -\lvert a_3\rvert^2  &= -\frac{p_1^4 + (4 - p_1^2)^2 \lvert \zeta\rvert^2 + 2 p_1^2 (4 - p_1^2) \RE{\zeta}}{4 (3 - 2 \beta)^2}.
\end{align*}
    A simple computation yields that
\begin{align*}
     T_{3,1}(f) = 1 &+  \frac{1}{4 (3 - 2 \beta)^2 (2 - \beta)^2}\bigg( p_1^4 (8 - 4 \beta - \beta^2) -8 p_1^2 (3 - 2 \beta)^2 \\
      &- (4 - p_1^2)^2 (2 - \beta)^2 \lvert \zeta\rvert^2 + 2 p_1^2 (4 - p_1^2) (2 - \beta^2) \RE{\zeta}\bigg)\\
              &=: g(p_1, \zeta, \RE(\zeta)).
\end{align*}
   Since the classes $\mathcal{A}_\beta$ and $\mathcal{P}$ are rotationally invariant, we can take $p=p_1\in [0,2]$. Using $\RE(\zeta)\geq - \lvert \zeta \rvert$ with notation $\lvert \zeta \rvert =y$, we have $g(p_1, \lvert \zeta \rvert , \RE{\zeta} ) \geq g_1(p, y)$, where
\begin{align*}
    g_1(p,y)= 1 +  \frac{1}{4 (3 - 2 \beta)^2 (2 - \beta)^2}\bigg( & p^4 (8 - 4 \beta - \beta^2) -8 p^2 (3 - 2 \beta)^2  \\
      &- (4 - p^2)^2 (2 - \beta)^2 y^2 - 2 p^2 (4 - p^2) (2 - \beta^2) y\bigg).
\end{align*}
     Also, note that
     $$\frac{\partial g_1(p,y)}{\partial y} = -\frac{2 (4 - p^2)^2 y (2 - \beta)^2 + 2 p^2 (4 - p^2) (2 - \beta^2)}{4 (3 - 2 \beta)^2 (2 - \beta)^2}<0$$
     for all $p\in [0,2]$ and $\beta \in [0,1]$. Hence $g_1(p,y)$ is a decreasing function of $y$ with $g_1(p,y)\leq g_1(p,1)=:g_2(p)$. Minimum of $g_2(p)$ is the lower bound of $\det T_{3,1}(f)$. The equation $g_{2}'(p)=0$ gives the following critical points
     $$ p^{(1)}=0, \quad  p^{(2)}= \pm \sqrt{\frac{(2 \beta^2 - 8 \beta +7)}{(2-\beta^2)}}.$$
     Using the basic calculus rule, it can be easily observed that the function $g_2(p)$ attains its minimum value at $p^{(2)}$ as $g''(p^{(2)})>0$ for all $\beta \in [0,1]$. Thus
     $$ \det T_{3,1}(f) \geq g_2(p^{(2)}) =  1 - ( 4 \beta - 9)/(  \beta^4 - 4 \beta^3 + 2 \beta^2 + 8 \beta -8 ).$$
     To show the sharpness consider the function $\tilde{f}_3 \in \mathcal{A}$ given by
     $$ \beta \frac{\tilde{f}_3(z)}{z} + (1- \beta) \tilde{f}'_3(z) = \frac{1 - z^2 }{ 1 - z \sqrt{(2 \beta^2 - 8 \beta +7)/(2 - \beta^2)} + z^2}. $$
     For $\tilde{f}_3(z) = z  + \sum_{n=2}^\infty a_n z^n $, we have
     $$ a_2 =\frac{1}{2 - \beta} \sqrt{\frac{ 2 \beta^2 - 8 \beta + 7}{ 2 - \beta^2}}, \quad a_3 = \frac{1 - 2\beta }{ 2 - \beta^2}$$
     and $ T_{3,1}(\tilde{f}_3) = 1 - ( 4 \beta - 9)/(  \beta^4 - 4 \beta^3 + 2 \beta^2 + 8 \beta -8 ).$
\end{proof}
\begin{remark}
     For $\beta =0$ in Theorem \ref{LHT}, we obtain $\det T_{3,1}(f) \geq -1/8$ for $f\in \mathcal{R}$ \cite[Example 2.4]{VkumarT}.
\end{remark}
\section{Coefficient Difference}
    Robertson \cite{Robertson} proved that $ 3 \lvert a_{n+1} - a_n \rvert \leq (2 n + 1) \lvert a_2 -1 \rvert$ for the class of convex functions. Recently, Li and Sugawa \cite{Li} obtained the bound of $\lvert a_{n+1} - a_n \rvert$ for particular choices of $n$ for the class of convex function with fixed second coefficient.
    In this section, we find the the bound of $\lvert a_{n+1}^N - a_n^N \rvert$ $(N \in \mathbb{N})$ depending on the second coefficient for $f\in \mathcal{A}_\beta$. In fact, it is more convenient to express our result in terms of $p_1=p$, applying the correspondence
    $$ (2-\beta ) a_2 =p_1 =p.$$
    To make the results more legible, we define the class $\mathcal{A}_\beta(p)$, $p \in [-2,2]$  as follows
    $$ \mathcal{A}_\beta (p) = \{ f\in \mathcal{A}_\beta : f''(0)=p\}.   $$
    Clearly,
   $$ \bigcup_{-2 \leq p \leq 2} \mathcal{A}_\beta (p) \subset \mathcal{A}_\beta \quad \text{and} \quad  \bigcup_{-2 \leq p \leq 2}\mathcal{A}_\beta (p) \neq \mathcal{A}_\beta . $$
   The following lemmas are used to establish our main results.
\begin{lemma}\label{brownl}\cite{Brown}
   If $p(z) = 1+ \sum_{n=1}^\infty p_n z^n \in \mathcal{P}$, then the following estimate holds:
   $$ \lvert p_{n+1}^N -p_n^N \rvert \leq 2^N \sqrt{2 - 2^{1-N} \RE (p_1^N)} \quad (N \in \mathbb{N}). $$
   Equality holds for the function $(1+ e^{i \alpha } z)/(1- e^{i \alpha } z),$ where $\alpha = \cos^{-1}(b/2)$ and $\RE{p_1} = 2b.$
\end{lemma}
\begin{lemma}\cite{Lecko}\label{llemma}
   Fix $\zeta \in \bar{\mathbb{D}}$. If $p(z) = 1+\sum_{n=1}^\infty p_n z^n \in \mathcal{P}$, then
\begin{align*}
    \lvert \xi p_{n+1} - p_n \rvert \leq
        \frac{2(1 - \lvert \xi \rvert^n)\left(1 + \lvert \xi \rvert^2 -  \RE (\xi p_1) \right)}{1 - \lvert \xi \rvert} + \lvert 2 - \xi p_1 \rvert \lvert \xi \rvert^n \;\; &\text{for} \;\; \lvert \xi \rvert <1.
\end{align*}
    The bounds are sharp for $p(z)=(1+z)/(1-z)$.
\end{lemma}
   According to Komatu \cite{Komatu}, if $p(z)=1+\sum_{n=1}^\infty p_n z^n$ and $q(z) = 1+ \sum_{n=1}^\infty q_n z^n$ both are the members of $\mathcal{P}$, then the weighted Hadamard product, $f * g $, also belongs to $\mathcal{P}$, where
    $$ f* g =1 + \sum_{n=1}^\infty \frac{p_n q_n}{2} z^n.$$
    Let us define  $F_j (z) = F_{j-1} * p(z)$ for $j \in \mathbb{N}$ with $F_0(z)= p(z)$, then using the above result, we have $F_j \in \mathcal{P}$. Particulary, for $N \in \mathbb{N}$, the function
    $$ F_{N-1} (z) = 1 + \sum_{n=1}^\infty \frac{p_{n}^N}{2^{N-1}} z^n  \in \mathcal{P}.$$
    Replacing $p(z)$ in Lemma \ref{llemma} by $F_{N-1}$, the result is as follows:
\begin{lemma}\label{llemma2} Fix $\xi \in \bar{\mathbb{D}}$ and $N\in \mathbb{N}$. If $p(z) = 1+ \sum_{n=1}^\infty p_n z^n$, then
\begin{align*}
   & \lvert \xi p_{n+1}^N - p_{n}^N \rvert \leq  \\
       & \frac{2 (1 - \lvert \xi \rvert^n)\left(2^{N-1} +2^{N-1} \lvert \xi \rvert^2 -  \RE (\xi p_{1}^N) \right)}{1 - \lvert \xi \rvert} + \lvert 2^{N} - \xi p_{1}^N \rvert \cdot \lvert \xi \rvert^n  &\text{for} \;\; \lvert \xi \rvert <1,
\end{align*}
    Equality holds for the function $(1+z)/(1-z)$.
\end{lemma}
\begin{theorem}\label{thm7}
    If $f\in \mathcal{A}_\beta(p)$, then the following inequalities hold:
\begin{equation}\label{diff}
\begin{aligned}
     & \lvert  a_{n+1}^N -  a_n^N \rvert \leq \\
     &\left\{ \begin{array}{ll}
     \dfrac{2 (\sigma^n -\mu^n  ) ( 2^{N-1} \sigma^2 + 2^{N-1} \mu^2 - \sigma \mu p^N  )}{(\sigma - \mu) \sigma \mu^{n + 1}}  + \dfrac{\sigma^n \lvert 2^N \mu - \sigma p^N  \rvert}{ \sigma \mu^{n + 1}}; &\beta\in [0,1), \\ \\
      \dfrac{2^N \sqrt{2 - 2^{1-N}  p^N}}{\sigma}; & \beta=1,
     \end{array}
     \right.
\end{aligned}
\end{equation}
   where $\sigma =  (n - (n-1) \beta)^N$ and $\mu = (n + 1 - n \beta)^N$.  Bounds for $\beta \in [0,1)$ is sharp for $p=2$ whereas for $\beta =1$, bound is sharp for odd $N$ and $p=-2$.
\end{theorem}
\begin{proof}
   For $f\in \mathcal{A}_\beta(p)$, from (\ref{three}), we have
\begin{align*}
   (n- (n-1) \beta)^N \lvert a_{n+1}^N - a_{n}^N \rvert &= \bigg\lvert \bigg(\frac{n- (n-1) \beta}{(n+1)- n \beta} \bigg)^N p_n^N - {p_{n-1}^N} \bigg\rvert. \\
\end{align*}
   From Lemma \ref{llemma2} with $ \left((n- (n-1) \beta)/((n+1)- n \beta)\right)^N =: \xi$, bound in (\ref{diff}) for $\beta \in [0,1)$ follows since $\xi \in (0,1)$ whenever $\beta \in (0,1)$. For $\beta =1$ we have $\xi =1$. Bounds for $\beta =1$ are obtained using Lemma \ref{brownl} .

   To show the sharpness for $\beta \in [0,1)$, consider the function $\tilde{f}(z)$ given in (\ref{seven}). As for $\tilde{f}$, we have
   $$ \lvert a_{n+1} - a_{n} \rvert = \frac{2^{N}}{(n - (n - 1) \beta)^N} \left\lvert \frac{(n - (n - 1)\beta)^N}{(n +1 - n\beta)^N} -1  \right\rvert, $$
   which is same as in (\ref{diff}) for $p=2$. In case of $\beta =1$, for the function $\tilde{f}(-z)$, we have
   $$  \lvert a_{n+1} -a_n \rvert = \frac{2^{N+1}}{(n- (n-1) \beta)^N},$$
   which coincides with the bounds in (\ref{diff}) for odd $N$ and $p=-2$.
\end{proof}
    For $N=1$, Theorem \ref{thm7} yields the following bounds:
\begin{corollary}
    If $f\in \mathcal{A}_\beta(p)$ is of the form (\ref{zero}), then
\begin{align*}
     \lvert  a_{n+1} -  a_n \rvert \leq
     \left\{ \begin{array}{ll}
     \dfrac{2 (\sigma^n -\mu^n  ) ( \sigma^2 +  \mu^2 - \sigma \mu p  )}{(\sigma - \mu) \sigma \mu^{n + 1}}  + \dfrac{\sigma^n \lvert 2 \mu - \sigma p  \rvert}{ \sigma \mu^{n + 1}}; &\beta\in [0,1), \\ \\
      \dfrac{2 \sqrt{2 -   p}}{\sigma}; & \beta=1.
     \end{array}
     \right.
\end{align*}
\end{corollary}
    The class $\mathcal{A}_\beta$ reduces to the class $\mathcal{R}$ for $\beta =0$.  Let us take corresponding class $\mathcal{R}(p)= \{ f \in \mathcal{R} : f''(0)=p\}$. Theorem \ref{thm7} gives the following result for the class $\mathcal{R}(p)$ when $\beta =0$.
\begin{corollary}
   If $f\in \mathcal{R}(p)$ is of the form (\ref{zero}), then the following sharp bounds hold:
   $$  \lvert  a_{n+1}^N -  a_n^N \rvert \leq
     \frac{2 (\sigma^n -\mu^n  ) ( 2^{N-1} \sigma^2 + 2^{N-1} \mu^2 - \sigma \mu p^N  )}{(\sigma - \mu) \sigma \mu^{n + 1}}  + \frac{\sigma^n \lvert 2^N \mu - \sigma p^N  \rvert}{ \sigma \mu^{n + 1}}.$$
\end{corollary}
\section{Growth Theorem and Bohr Phenomenon}\label{Bohr}
\begin{theorem}\label{growth}
     If $f\in \mathcal{A}_\beta$ is of the form (\ref{zero}), then for $\lvert z \rvert \leq r$, the following hold:
\begin{enumerate}[(i)]
  \item $ -\dfrac{\tilde{f}(-r)}{r} \leq \RE \bigg(\dfrac{f(z)}{z} \bigg) \leq  \dfrac{\tilde{f}(r)}{r},$  
  \item $ -\tilde{f}(-r) \leq  \lvert f(z) \rvert \leq \tilde{f}(r)  , $\\
\end{enumerate}
    where $\tilde{f}(z)$ is given by (\ref{seven}). All these estimations are sharp.
\end{theorem}
\begin{proof}
  (i) Let $f\in \mathcal{A}_\beta.$ Consider
   $p(z)={f(z)}/{z}$, then we have
   $$ \RE (p(z)+(1-\beta) z p'(z)) >0.$$
   It can be viewed as
   $ p(z)+ (1 - \beta) z p'(z) \prec {(1+z)}/{(1-z)}.$
   Further, by Hallenbeck and Rusheweyeh~\cite[Theorem 3.1b]{MM}, it follows that
  $$ p(z) \prec q(z) \prec  \frac{1+z}{1-z},$$
  where $q(z)$ is convex and best dominant,
   given by
   \begin{align*}
   q(z) &= \frac{1}{(1-\beta)z^{(\frac{1}{1-\beta})}} \int_0^z \left( \frac{1+t}{1-t} \right) t^{(\frac{1}{1-\beta}-1)} dt\\
      &= \frac{\tilde{f}(z)}{z},
    \end{align*}
   where $\tilde{f}(z)$ is defined in (\ref{seven}). Since $q(z)$ is convex and all coefficients are real for $\beta \in [0,1]$, therefore image domain of $\mathbb{D}$ under the function $q(z)$ is symmetric with respect to real axis and
   $$  q(-r) \leq  \RE (q(z)) \leq q(r), \quad \lvert z\rvert =r <1. $$
   As $p(z) =f(z)/z \prec q(z)$, so required bound of $\RE (f(z)/z)$ follows. This completes the first part. Sharpness of the bounds follow as $q(z)$ is the best dominant.
 (ii) From \cite[Lemma 4.10]{FB1}, $f \in \mathcal{A}_\beta$ if and only if
\begin{equation}\label{last}
      f(z) = z \int_0^1 p(t^{1-\beta} z) dt,
\end{equation}   where $p\in \mathcal{P}$.  
        Using the well known bound $ \lvert p(z) \rvert \leq (1+r)/(1-r)$ of Carath\'{e}odory functions, we have
\begin{equation*}
      \lvert f(z) \rvert \leq r \int_0^1 \frac{1 + r t^{1-\beta}}{1 - r t^{1-\beta}} dt =\tilde{f}(r) ,
\end{equation*}
     Now, we proceed for the lower bound of $\lvert f(z) \rvert$. After solving the integration in (\ref{last}) for $p(z)=(1+z)/(1-z)$, we get
       $$ f(z) = z ( -1 + 2 H(z)), $$
   where
   $$ H(z)=  {}_2 F_1 \bigg[1, \frac{1}{1-\beta}, \frac{2- \beta}{1-\beta},z \bigg].$$
   Thus for $z = r e^{i \theta}$,
\begin{equation}\label{nine}
    \lvert f(z)\rvert = \lvert z (-1 + 2 H(z))  \rvert  \geq  \min_{ \theta \in [0, 2 \pi]} g(\theta)
\end{equation}
   where
   $$ g(\theta) = \sqrt{  \RE( r e^{i \theta} ( -1 + 2 H( r e^{i \theta})))^2 + \IM (r e^{i \theta} ( -1 + 2 H(r e^{i \theta})))^2 }, $$
     Since for different choices of $\beta$ in $[0,1)$, $H(z)$ reduces to different functions. For instance, when $\beta =0$, it becomes $-2 \log(1-z)/z$ and for $\beta=1/2$, it reduces to $-4 (z + \log(1 - z))/z^2$. By a simple calculation, we find that the function $g(\theta)$ is decreasing from $[0,\pi]$ and increasing from $[\pi, 2 \pi]$ for $r\in (0,1)$ and $\beta \in [0,1)$. Hence $g(\theta)$ attains its minimum value at $\theta = \pi$. Thus from (\ref{nine}), we get
\begin{align*}
    \lvert f(z) \rvert &\geq \lvert -r (-1 +2 H(-r))\rvert \\
                & = r (-1 + 2 H(-r))= -r \tilde{f}(-r),
\end{align*}
    which completes the proof.
    Bounds are sharp for the function $\tilde{f}(z)$.
\end{proof}
\begin{theorem}\label{BohrT}
    If $f\in \mathcal{A}_\beta$ is of the form (\ref{zero}), then for $m\in \mathbb{N} $
    $$ \lvert \omega(z^m) \rvert + \sum_{n=2}^\infty \lvert a_n z^n \rvert \leq d(0, \partial f(\mathbb{D}))$$
    in $\lvert z \rvert \leq r^*$, where $r^*$ is the smallest positive root of
\begin{equation}\label{bohr}
    r^m + \tilde{f}(r) -r + \tilde{f}(-1)=0.
\end{equation}
   The radius $r^*$ is sharp.
\end{theorem}
\begin{proof}
   Let $f\in \mathcal{A}_\beta$, then by Theorem \ref{growth}, the Euclidean distance between $f(0)=0$ and the boundary of $f(\mathbb{D})$ satisfies
   $$ d(0, \partial f(\mathbb{D})) \geq \lim_{ r \rightarrow 1} \lvert f(z) \rvert = - \tilde{f}(-1).  $$
   Let $\lvert z \rvert \leq r$. Now using (\ref{two}) with the above inequality, we have
\begin{align*}
    \lvert \omega(z^m) \rvert + \sum_{n=2}^\infty \lvert a_n z^n \rvert  &\leq  r^m + \sum_{n=2}^\infty \left( \frac{2}{n-\beta (n-1)} \right) r^n, \\
                                                               & = r^m + \tilde{f}(r) -r    \\
                                                                &\leq - \tilde{f}(-1) \leq d(0, \partial f(\mathbb{D})).
\end{align*}
    which is true in $\lvert z \rvert =r \leq r^*$, where $r^*$ is the root of $H(r)=r( r^{m-1} -1) + \tilde{f}(r) + \tilde{f}(-1)$. Note that, $H(0) = \tilde{f}(-1) < 0$ and $H(1)= \tilde{f}(1)+ \tilde{f}(-1)>0 $ for all $\beta \in [0,1]$, therefore by the Intermediate value property for continuous functions there must exist a $r^* \in (0,1)$ such that $H(r^*)=0$.

    Sharpness holds for the functions $\tilde{f}(z)$ and $\omega(z) = z$. Since at $z =r^*$,
\begin{align*}
     \lvert \omega(z^m) \rvert + \sum_{n=2}^\infty \lvert a_n z^n \rvert  &= (r^*)^m + \sum_{n=2}^\infty \frac{2}{n- (n-1)\beta}(r^*)^n   \\
                                                                         &= (r^*)^m + \tilde{f}(r^*)-r^* = -\tilde{f} (-1).
\end{align*}
   Hence the radius is sharp.
\end{proof}
    For $w(z)=z$ and $m=1$, Theorem \ref{BohrT} gives the following Bohr-radius for the class $\mathcal{A}_\beta.$
\begin{corollary} If $f\in \mathcal{A}_\beta$, then
       $ \lvert z \rvert + \sum_{n=2}^\infty \lvert a_n z^n \rvert \leq d(0, \partial f(\mathbb{D}))$
    in $\lvert z \rvert \leq r_b$, where $r_b$ is root of
    $ \tilde{f}(r) + \tilde{f} (-1)=0.$
   The radius $r_b$ is sharp.
\end{corollary}
   For various values of $\beta \in [0,1]$, the root $r_b$ is shown in Figure 1 and Table 1. \\

\begin{figure}[h]
\centering
\captionsetup{justification=centering,margin=2cm}

  \includegraphics[width=8cm, height=4cm]{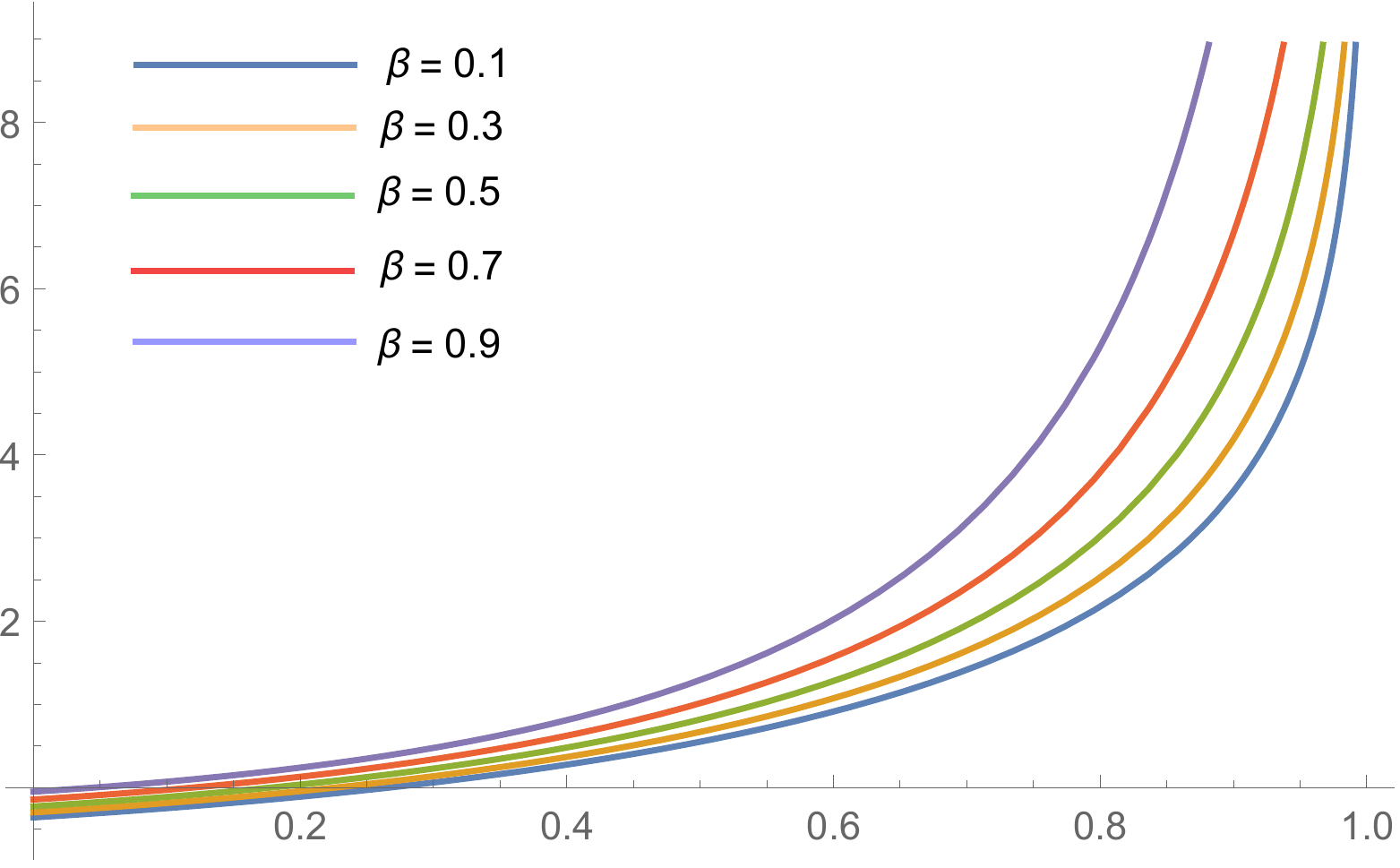}
 \caption{}
 \label{fig:mesh1}
\end{figure}

\begin{table}[h]
\begin{center}
\caption{Radius $r^*$ for various choices of $\beta$}\label{tab1}%
\begin{tabular}{@{}cccccccc@{}}
\toprule
$\beta$  & 0.1 & 0.2 & 0.3 & 0.5 & 0.7 & 0.8  &0.9   \\
\midrule
  $r_b$  & 0.267139 & 0.24766  &  0.22655   &  0.178366  & 0.119726 &  0.085113 & 0.0457777 \\ 
\toprule
\end{tabular}
\end{center}
\end{table}
\begin{theorem}
    If $f\in \mathcal{A}_\beta$, then
\begin{equation}\label{bohr-rogo}
     \lvert f(z^m) \rvert+ \sum_{k=N}^\infty \lvert a_k z^k \rvert \leq d(0, \partial f(\mathbb{D}))
\end{equation}
    hold for $ \lvert z \rvert = r \leq  r_N$, where $r_N$ is the root of the equation
    $$ \tilde{f}(r^m) + \tilde{f} (r) - \hat{f}(r) + \tilde{f}(-1) =0, $$
    with
\begin{align*}
     \hat{f}(r) = \left\{
    \begin{array}{ll}
    0 & N=1, \\
    r & N=2, \\
    r+ \sum_{n=2}^{N-1} \frac{2}{n-(n-1)\beta} r^n & N\geq 3.
    \end{array}
    \right.
\end{align*}
    The radius is sharp.
\end{theorem}
\begin{proof}
   Suppose $f \in \mathcal{A}_\beta$, then from (\ref{one}) and  Theorem \ref{growth}, we have
\begin{align*}
    \lvert f(z^m)\rvert + \sum_{k=N}^\infty \lvert a_k z^k \rvert &\leq \tilde{f}(r^m)  + \sum_{n=N}^\infty \frac{2}{n - (n-1) \beta} r^n \\
                                                            & = \tilde{f}(r^m) - \hat{f}(r) + \tilde{f}(r)\\
                                                            & \leq - \tilde{f}(-1)\\
                                                            &\leq d(0, \partial f(\mathbb{D}))
\end{align*}
    holds in $\lvert z \rvert = r_N$, where $r_N$ is the root of
    $$ G(r):= \tilde{f}(r^m) - \hat{f}(r) + \tilde{f}(r) + \tilde{f}(-1)=0.$$
    Since $G(0)=\tilde{f}(-1)<0 $ and $G(1)= (\tilde{f}(1) - \hat{f}(1))+(\tilde{f}(1) + \tilde{f} (-1))>0$, therefore there exist a $r_N \in (0,1)$ such that (\ref{bohr-rogo}) holds. Note that, for the function $\tilde{f}(z)$ at $\lvert z \rvert =r_N $,
\begin{align*}
    \lvert f(z^m) \rvert + \sum_{k=N}^\infty \lvert a_k z^k \rvert &= \tilde{f} ((r_N )^m)  + \sum_{n=N}^\infty \frac{2}{n - (n-1) \beta} (r_N )^n \\
                                                                   &= -\tilde{f}(-1),
\end{align*}
    which proves the sharpness of radius.
\end{proof}
\section{Declarations}
\subsection*{Funding}
The work of the Surya Giri is supported by University Grant Commission, New-Delhi, India  under UGC-Ref. No. 1112/(CSIR-UGC NET JUNE 2019).
\subsection*{Conflict of interest}
	The authors declare that they have no conflict of interest.
\subsection*{Author Contribution}
    Each author contributed equally to the research and preparation of manuscript.
\subsection*{Data Availability} Not Applicable.
\noindent

\end{document}